\documentclass{amsart}

\usepackage{amsmath,amsfonts,amssymb,amsthm}
\usepackage{hyperref}
\usepackage{mathtools}
\usepackage{pdfpages}
 \usepackage{textcomp}
 \usepackage{ragged2e}
 \usepackage{blindtext}
 
 \usepackage{thmtools}
 
\usepackage[capitalize]{cleveref}
\usepackage{stmaryrd}
\usepackage{tikz-cd}

\usepackage{graphicx}

\newcommand{\hookdownarrow}{\mathrel{\rotatebox[origin=c]{-90}{$\hookrightarrow$}}}

\newtheorem{lem}{Lemma}[section]
\newtheorem{teo}[lem]{Theorem}

\newtheorem{pro}[lem]{Proposition}
\newtheorem{cor}[lem]{Corollary}
\newtheorem{claim}[lem]{Claim}

\newtheorem*{con*}{Conjecture}

\newtheorem{Conj}{Conjecture}
\newtheorem{Question}[Conj]{Question}
\newtheorem{Questions}[Conj]{Questions}

\theoremstyle{definition}
\newtheorem{exa}[lem]{Example}

\theoremstyle{remark}
\newtheorem*{rem*}{Remark}

\newcommand{\argu}{\hbox to 7truept{\hrulefill}}

\DeclareMathOperator{\B}{Bohr}
\DeclareMathOperator{\A}{A}

\DeclareMathOperator{\adj}{adj}

\DeclareMathOperator{\SL}{SL}\DeclareMathOperator{\GL}{GL}

\DeclareMathOperator{\im}{Im}

\DeclareMathOperator{\Hom}{Hom}

\DeclareMathOperator{\Aut}{Aut}

\DeclareMathOperator{\Spin}{Spin}

\DeclareMathOperator{\Mat}{Mat}

\DeclareMathOperator{\PGL}{PGL}
\DeclareMathOperator{\SO}{SO}

\DeclareMathOperator{\Max}{Max}

\newcommand{\R}{\mathbb{R}}

\newcommand{\Z}{\mathbb{Z}}
\newcommand{\F}{\mathbb{F}}

\newcommand{\N}{\mathbb{N}}
\newcommand{\CC}{\mathbb{C}}
\newcommand{\Q}{\mathbb{Q}}

\newcounter{marcocomments}

\newcounter{andreicomments}

 \date{\today}

\begin{document}

\title[Some remarks on Grothendieck pairs]
{Some remarks on Grothendieck pairs\vspace{0.5cm}
\rightline{ \tiny \it \underline
{to Slava Grigorchuk}} 
\rightline{ \tiny \it \underline{with admiration and affection}}
}
\author{Andrei Jaikin-Zapirain}
 
 \address{Departamento de Matem\'aticas, Universidad Aut\'onoma de Madrid \and  Instituto de Ciencias Matem\'aticas, CSIC-UAM-UC3M-UCM}
\email{andrei.jaikin@icmat.es}

\author{Alexander  Lubotzky}
\address{Weizmann Institute of Science, Rehovot, Israel}    \email{alex.lubotzky@mail.huji.ac.il}

\begin{abstract}

We revisit the paper of Alexander Grothendieck where he introduced  Grothendieck pairs and discuss the relation between profinite rigidity and  left/right Grothendieck rigidity. We also show that   various groups are left and/or right Grothendieck rigid and, in particular, all ascending HNN extensiona  of finitely generated free groups are  right Grothendieck  rigid. Along the way we present a number of questions and suggestions for further research.
 
\end{abstract}

\maketitle

\section{Introduction}
Let $\phi:\Lambda \hookrightarrow \Gamma$ be an embedding between two finitely generated residually finite groups. We say that $\phi$ is a {\bf Grothendieck pair} if the induced homomorphism between the profinite completions $\widehat \phi: \widehat \Lambda \to \widehat \Gamma$ is an isomorphism We say that a Grothendieck pair is {\bf trivial} if $\phi$ is an isomorphism, i.e., $\phi(\Lambda)=\Gamma$.

Grothendieck pairs were introduced in \cite{Gr70} by Grothendieck.  His paper became well-known because of a proposed  question as to whether a Grothendieck pair of two finitely presented groups is always trivial. 
The first example of finitely generated non-trivial Grothendieck pairs were given by Platonov and Tavgen\textquotesingle\  \cite{PT86}. This has been followed  by many other constructions (\cite{BL00,Py04, KS23}), and eventually  Bridson and Grunewald \cite{BG04} gave an example were both groups were finitely presented. Moreover, the paper of Grothendieck proposed another conjecture concerning  the groups $\mathrm {cl}_A(\Gamma)$ which could lead to proving the main conjecture. However, this conjecture also  turned out to be wrong (see \cite{Lu79, LV19} for more details). Much work to build more counterexamples has overshadowed some of Grothendieck's remarkable results obtained in his paper. In this article we want to revisit  them and put them in the context of the recent ongoing work on profinite rigidity.

Recall that a  finitely generated residually finite group $\Lambda$ is called {\bf profinitely rigid} if for any finitely generated residually finite group $\Gamma$ whose profinite competion is isomoprhic to the one of $\Lambda$, $\Gamma\cong \Lambda$.
A finitely generated residually finite group $\Lambda$ will be called 
{\bf   left Grothendieck  rigid} ({\textbf L}GR for short) if whenever $\phi:\Lambda \hookrightarrow \Gamma$ is a
 Grothendieck pair,  $\phi$ is an isomorphism. Similarly, $\Gamma$ is called {\bf   right Grothendieck  rigid}  ({\textbf R}GR for short) if whenever 
$\phi:\Lambda \hookrightarrow \Gamma$ is a Grothendieck pair, it is trivial.
The notion of right Grothendieck rigid has previously  appeared in the literature as Grothendick rigid (see, for example, \cite{ReAndrews, ReICM}). We present  examples of {\textbf L}GR and {\textbf R}GR groups  in Sections \ref{GLRsection} and  \ref{LRGR} and in Sections   \ref{LRGR} and \ref{profinite}   we also discuss  relations between these three notions. Among  the new results we will show that the mapping tori of a finitely generated free group is {\textbf R}GR (see   Section \ref{HNN}).

In his paper Grothendieck introduced the   following class  $\mathcal C$ of groups: $G\in  \mathcal C$ if for any Grothendieck pair $\phi:\Lambda \hookrightarrow \Gamma$, the induced map 
$$\phi^{\#}_G: \Hom (\Gamma, G)\to \Hom (\Lambda, G), \ \tau\mapsto \tau\circ \phi,$$  is a bijection. The following theorem summarizes the results of Grothendieck that we want to emphasize.
\begin{teo}\label{grothendieck}
The following holds.

\begin{enumerate}
\item The class $\mathcal C$  is closed under comensurabilty, inverse limits and direct products and contains all nilpotent groups.
\item If $A$ is a commutative  ring and $G$ is an affine group scheme of finite type over $A$ then $G(A)\in \mathcal C$.
\item Compact Hausdorff  groups are in $\mathcal C$.
\item Let $\Lambda$ be a  finitely generated residually finite group. If $\Lambda \in \mathcal C$, then $\Lambda$ is {\textbf L}GR.

\end{enumerate}
\end{teo}

Note that part (2) of the theorem is a remarkable super-rigidity result, implying that for every Grothendieck pair $\Lambda\hookrightarrow \Gamma$, every finite dimensional representation of $\Lambda$ can be extended to $\Gamma$. This was a crucial ingredient in \cite{BL00}.
Here are some consequences of Grothendieck's theorem.
 \begin{cor}\label{leftGR}
 The following groups are left  Grothendieck rigid:
 
 \begin{enumerate}
 \item finitely generated free groups;
 \item surface groups;
 \item $S$-arithmetic groups.
 \end{enumerate}
 \end{cor}
 Given a group $\Gamma$, its {\bf Bohr compactifications} is   a pair   ($\B(\Gamma), \beta)$ consisting of a compact (Hausdorff) group $\B(\Gamma)$
and a  homomorphism $\beta : \Gamma \to \B(\Gamma)$
  satisfying the following universal property: for every
compact group  $K$ and every 
homomorphism $\alpha : \Gamma \to K$, there exists a unique continuous homomorphism $\widetilde \alpha: \B(\Gamma)\to K$ such that $\alpha=\widetilde \alpha \circ \beta$. The pair $(\B(\Gamma), \beta)$ is unique in the following sense: if
$(L,\beta^\prime$) is a pair consisting of a compact group $L$ and a  
homomorphism $\beta^\prime : \Gamma \to L$  satisfying the same
universal property, then there exists  an isomorphism $\alpha : \B(\Gamma) \to L$ of topological
groups such that $\beta^\prime \circ \alpha=\beta$.

  The {\bf proalgebraic completion} $\A(\Gamma)$ of a group $\Gamma$, also called the Hochschild\-–Mostow group of  $\Gamma$
is the  proaffine complex algebraic  group $A(\Gamma)$ with a homomorphism
$\alpha: \Gamma\to \A(\Gamma)$ such that for every representation $\rho:\Gamma\to \GL_n(\CC)$   there is a unique algebraic representation $\widetilde  \rho$ such that  $\widetilde  \rho\circ \alpha=\rho$. As in the case of the Bohr compactification the pair $(\A(\Gamma),\alpha)$ is unique in a canonical way.

\begin{cor}\label{bohr}
Let $\phi:\Lambda \hookrightarrow \Gamma$ be a Grothendieck pair. Then 
\begin{enumerate}
\item $\Phi$ induces an isomorphism $\B(\phi):\B(\Lambda)\to \B(\Gamma)$;
\item $\Phi$ induces an isomorphism $\A(\phi): \A(\Lambda)\to \A(\Gamma)$.
\end{enumerate}
\end{cor}
In Section \ref{profinite} we present an example that shows that the isomorphism of profinte completions does not imply the isomorphism of Bohr compactifications. We do not know if this is the situation also with the proalgebraic completions.

\section*{Acknowledgments}

The work of  A. J.-Z.  is partially supported by the grant  PID2020-114032GB-I00 of the Ministry of Science and Innovation of Spain  and by the ICMAT Severo Ochoa project  CEX2019-000904-S4. The work of A. L. was supported by the European Research Council (ERC) under the European Union’s Horizon 2020 (grant agreement No 882751).

   This note grew out from  the workshop on ``Profinite Rigidity" organized by the Institute of Mathematics of Madrid (ICMAT). We would like to thank the institute for this opportunity to work together.  We are   grateful to Nir Avni and Alan Reid for  fruitful discussions.

 \section{Commutative algebra preliminaries}
 Let  $R$  be a finitely generated  commutative ring. Denote by  $\widehat R$ the profinite completion of $R$ and by $\Max R$   the set of its maximal ideals. 

Let  $\mathbf m\in \Max R$.  Observe that, since the field $R/\mathbf m$ is finitely generated as a ring, it is of positive characteristic and thus, by Hilbert's Nullstellensatz \cite[Corollary 5.24]{AMbook}, it is also finite.

 Let $ R_{ {\mathbf m}}$ denote the localization of $R$ at the maximal  ideal $\mathbf m$,  and  let $R_{\widehat{\mathbf m}}=\varprojlim R /\mathbf m^i$ be the $\mathbf m$-adic completion of $R$. By  \cite[Corollary 10.20]{AMbook}, the natural homomorphism $R_ {\mathbf m}\to R_{\widehat{\mathbf m}}$ is injective.  On the other hand, by \cite[Proposition 3.9]{AMbook}, the natural isomorphism $R\to \prod_{\mathbf m\in \Max (R)} R_{ {\mathbf m}}$ is also injective. Thus, we obtain that the map $R\to \prod_{\mathbf m\in \Max (R)} R_{\widehat {\mathbf m}}$ is injective, and so, $R$ is residually finite.

Recall that an $R$-module is {\bf faithfully flat} if taking the tensor product with a sequence produces an exact sequence if and only if the original sequence is exact. 
 \begin{pro}(\cite[Exercise 10.7]{AMbook}) \label{ff}  Let  $R$  be a finitely generated  commutative ring and $\mathbf m\in \Max(R)$.
 Then $R_{\widehat{\mathbf m}}$ is  faithfully flat as a $R_{ {\mathbf m}}$-module.
 \end{pro}
 The following corollary is a standard consequence of faithfully flatness (see, for example, \cite[Theorem III.6.6]{Arnotes}).
 
 \begin{cor}\label{tensor}
  Let  $R$  be a finitely generated  commutative ring and $\mathbf m\in \Max(R)$. Then
  $$\{a\in R_{\widehat{\mathbf m}}\colon a\otimes 1=1\otimes a \in R_{\widehat{\mathbf m}}\otimes_{R_{ {\mathbf m}}}R_{\widehat{\mathbf m}}\}=R_{ {\mathbf m}}.$$
 \end{cor}
  \section{Grothendieck's theorem}\label{GLRsection}
  For convenience of the reader we reprove in this section  Grothendieck theorem following the main steps of Grothendieck's argument. The proof here might be slightly easier to read than the one in \cite{Gr70}.
   \begin{pro}\label{gr1}
Let $G\in \mathcal C$. If $H$ is commensurable with $G$, then $H\in \mathcal C$.
\end{pro}
\begin{proof}
Let  $\phi:\Lambda \hookrightarrow \Gamma$ be a Grothendieck pair. We identify the elements of $\Lambda$ with their images under the map $\phi$. Thus, we can view $\Lambda$ as a subgroup of $\Gamma$.
 We want to show that the induced map 
$$ \phi^{\#}_H:\Hom (\Gamma, H)\to \Hom (\Lambda, H), \ \tau\mapsto \tau\circ \phi,$$  is bijective. It is enough to consider two cases: $H$ is of finite index in $G$ or $G$ is of finite index in $H$.

In the first case, since $\phi_G^{\#}$ is injective,  $\phi_H^{\#}$ is injective as well. Let $\psi\in \Hom (\Lambda, H)$. Since $\phi_G^{\#}$ is surjective, there exists $\tau: \Gamma\to G$ such that $\psi=\tau\circ \phi$. Since $H$ is of finite index in $G$, $\im \tau \cap H$ is of finite index in $\im \tau$. Since $\widehat \phi$ is onto, $\im \psi$ is profinitely dense in $\im \tau$. Thus, since $\im \phi\le \im \tau \cap H$,  we obtain that
$\im \tau\le H$. Therefore, $\phi_H^{\#}$ is surjective.

Assume now that $G$ is a subgroup of $H$ of finite index. Since we have already proved the previous case, we can substitute $G$ by its core in $H$ and  assume that $G$ is also normal in $H$.

Let $\Lambda_1\le \Lambda $ ($\Gamma_1\le \Gamma$) be the intersection of all subgroups of $\Lambda$ ($\Gamma$) of index at most $|H:G|$. Then, it is clear that the restriction of $\phi$ on $\Lambda_1$, $\phi_1:\Lambda_1 \hookrightarrow \Gamma_1$ is also a Grothendieck pair. Thus $(\phi_1)_G^{\#}$ is a bijection. 

Let $\psi\in \Hom(\Lambda, H)$. Then, $\psi^{-1}(G)$ is of index at most $|H:G|$ in $\Lambda$, and so,   by our construction of $\Lambda_1$, $\Lambda_1\le \psi^{-1}(G)$. Hence $\psi(\Lambda_1)\le G$.  Consider the restriction of $\psi$ on $\Lambda_1$, $\psi_1:\Lambda_1 \to  G$.  Since $G\in \mathcal C$, there exists a unique homomorphism $\tau_1:\Gamma_1\to G$ such that $\psi_1=\tau_1\circ \phi_1$.  

If there are two homomorphisms $\tau, \tau^\prime: \Gamma\to G$ that extend $\psi$, then their restrictions on $\Gamma_1$ coincide with $\tau_1$. Since $\phi:\Lambda \hookrightarrow \Gamma$ is a Grothendieck pair, $\Gamma=\Gamma_1\Lambda$.  Thus,  $\tau=\tau^\prime$. This shows that $\phi_H$ is injective.

In order to show that $\phi_H$ is surjective, we use again that $\Gamma=\Gamma_1\Lambda$. For every  $a\in \Gamma_1$ and  every $b\in \Lambda$ we define 
$$\tau(ab)=\tau_1(a)\psi(b).$$ Since $\Gamma_1\cap \Lambda=\Lambda_1$ and $(\tau_1)_{|\Lambda_1}=\psi_1= \psi_{|\Lambda_1}$, $\tau$ is a well-defined. 
Let us show that $\tau$ is a  homomorphism.  Let  $b\in \Lambda$. Define two maps $\Gamma_1\to G$ by 
$$\alpha_1: a\mapsto \psi(b)\tau_1(a) \psi(b)^{-1}\textrm{\ and \ } \alpha_2: a\mapsto  \tau_1(bab^{-1}). $$
Their restrictions on $\Lambda_1$ coincide. Hence, since $G\in \mathcal C$, $\alpha_1=\alpha_2$. Therefore, for every  $a_1,a_2\in \Gamma_1$ and  every $b_1,b_2\in \Lambda$, we obtain that
\begin{multline*}
\tau(a_1b_1a_2b_2)=\tau(a_1b_1a_2b_1^{-1}b_1b_2)=\tau_1(a_1b_1a_2b_1^{-1})\psi(b_1b_2)=\\
\tau_1(a_1)\tau_1(b_1a_2b_1^{-1})\psi(b_1b_2)=\tau_1(a_1)\psi(b_1)\tau_1(a_2)\psi(b_1)^{-1}\psi(b_1b_2)=\\ \tau_1(a_1)\psi(b_1)\tau_1(a_2)\psi(b_2)=\tau(a_1b_1)\tau(a_2b_2).
\end{multline*}
Hence $\tau$ is a homomorphism which extends $\psi$,  and so $\phi_H$ is surjective.
\end{proof}
\begin{pro}\label{gr2}
The class $\mathcal C$ is closed under inverse limits and direct products and contains all nilpotent groups.
\end{pro}
\begin{proof} Let  $\phi:\Lambda \hookrightarrow \Gamma$ be a Grothendieck pair.

If $G$ is an inverse limit of $G_i$, then $\Hom (\Gamma, G)$ is a direct limit of $\Hom (\Gamma, G_i)$. Then since $\phi_{G_i}^{\#}$ are bijections, $\phi_{G}^{\#}$ is bijective.

If $G=G_1\times G_2$, then $\Hom (\Gamma, G)=\Hom (\Gamma, G_1)\times \Hom (\Gamma, G_2)$ and $\phi_G^{\#}=(\phi_{G_1}^{\#}, \phi_{G_2}^{\#})$. Thus, if $G_1$ and $G_2$ are in $\mathcal C$, then $G\in \mathcal C$ as well.

Finally, by \cite[Proposition 2]{PT90}, finitely generated nilpotent groups are {\textbf R}GR. Therefore,  $\phi$ induces an isomorphism $\phi_{n}:\Lambda/\gamma_n(\Lambda)\to \Gamma/\gamma_n(\Gamma)$, where $\gamma_n( \Lambda)$ denotes the $n$th term of the lower central series of $\Lambda$, and so for every nilpotent group $G$, $\phi_G^{\#}$ is a bijection.
\end{proof}
 Now we prove the second part of Theorem \ref{grothendieck}, which is the main part of the theorem.
\begin{pro}\label{gr3}
Let $\phi:\Lambda \hookrightarrow \Gamma$ be a Grothendieck pair, $A$  a commutative ring,  $G$ an affine group scheme of finite type over $A$ and $\psi: \Lambda\to G(A)$ a homomorphism. Then there exists a unique $\tau:\Gamma \to  G(A) $ such that $\psi=\tau\circ\phi$.
\end{pro}
\begin{proof} Since $\Lambda$ and $\Gamma$ are finitely generated, without loss of generality we may assume that $A$ is a finitely generated ring. In particualr, $A$ is residually finite.

Consider the following commutative diagram.
\begin{equation}\label{diagram1}\begin{array}{ccccc}
G (A)&\xleftarrow{\psi} & \Lambda & \xrightarrow{\phi} &\Gamma \\
\hookdownarrow &&\hookdownarrow &&\hookdownarrow \\
G (\widehat A)&\xleftarrow{\psi_{\widehat{A}}} & \widehat{ \Lambda }& \xrightarrow{\widehat{\phi}} &\widehat{\Gamma}
 \end{array}\end{equation}
First observe, that since $\widehat{\phi}$ is an isomorphism,  if $\tau$ exists, it is unique. We call this property  {\it the uniqueness of extensions}.

Let $B$ be the $A$-subring of $\widehat A$ generated by the  coordinates  of the elements  
from ${\psi_{\widehat{A}}}(\widehat{\phi}^{-1}(\Gamma))$. Denote by $\tau:\Gamma\to G (B)$ the  restriction of $ \psi_{\widehat{A}}\circ \widehat{\phi}^{-1}$ on $\Gamma$.
We want to show that the embedding $\alpha:A\hookrightarrow B$ is an isomorphism. Assume that it is not the case. 
Then, by \cite[Proposition 3.9]{AMbook}, there exists $\mathbf m\in \Max(A)$ such that the natural ring homomorphism  $\alpha_{\mathbf m}:A_{\mathbf m}\to B_{\mathbf m}$, induced by   $\alpha$, is not surjective. Observe that $\alpha_{\mathbf m}$ is injective by \cite[Proposition 3.9]{AMbook}. The homomorphism $\alpha_{\mathbf m}$ induces an embedding
$$\alpha_{\mathbf m, G}:G(A_{\mathbf m})\to G(B_{ \mathbf m}).$$

Consider the following commutative diagram.
 \begin{equation}\label{diagram2}\begin{array}{ccccc}
G (A_{\mathbf m})&\xleftarrow{\psi_{A_{\mathbf m}}} & \Lambda & \xrightarrow{\phi} &\Gamma \\
\hookdownarrow &&\hookdownarrow &&\hookdownarrow \\
G (A_{\widehat{\mathbf m}})&\xleftarrow{{\psi_{A_{\widehat{\mathbf m}}}}} & \widehat{ \Lambda }& \xrightarrow{\widehat{\phi}} &\widehat{\Gamma}
 \end{array}\end{equation}
 and let $\pi:\Gamma\to G (A_{\widehat{\mathbf m}}) $ be the restriction on  $\Gamma$ of the map $ \psi_{A_{\widehat{\mathbf m}}}\circ\widehat{\phi}^{-1}$.
 
 Consider the two maps $\alpha_1,\alpha_2:A_{\widehat{\mathbf m}}\to A_{\widehat{\mathbf m}}\otimes_{A_{ {\mathbf m}}} A_{\widehat{\mathbf m}}$ such that
 $$\alpha_1(a)=a\otimes 1 \textrm{\ and \ } \alpha_2(a)=1\otimes a.$$ For $i=1,2$, we obtain  the induced  maps 
 $\alpha_{i,G}: G(A_{\widehat{\mathbf m}})\to G( A_{\widehat{\mathbf m}}\otimes_{A_{ {\mathbf m}}} A_{\widehat{\mathbf m}})$. 
 We put $\pi_{i}=\alpha_{i,G}\circ \pi:\Gamma\to G (A_{\widehat{\mathbf m}}\otimes_{A_{ {\mathbf m}}} A_{\widehat{\mathbf m}})$. Since the restriction of $\pi_1$ and $\pi_2$ on $\Lambda$  coincide, the uniqueness of extensions (proved above)  implies that $\pi_1=\pi_2$. Thus, by Corollary \ref{tensor}, $\pi(\Gamma)\le G (A_{\mathbf m})$. Thus, we have constructed a representation $\pi:\Gamma\to G (A_{\mathbf m})$ that extends $\psi_{A_{\mathbf m}}$.

Consider the representation $\tau_{B_{\mathbf m}} :\Gamma\to G (B_{\mathbf m})$ induced by $\tau$. Observe that by our definition of $B$, the entries of the matrices from $\tau_{B_{\mathbf m}}(\Gamma)$ generate $B_{\mathbf m}$ over $A_{\mathbf m}$. By the other hand, the resrictions of $\tau_{B_{\mathbf m}}$ and $\alpha_{\mathbf m, G}\circ \pi$ on $\Lambda$ coincide. Hence, by  the uniqueness of extensions $\tau_{B_{\mathbf m}}=\alpha_{\mathbf m, G}\circ \pi$. This means that $B_{\mathbf m}=\alpha_{\mathbf m}(A_{\mathbf m})$, and so $\alpha_{\mathbf m}$ is onto. A contradiction. Hence $B=A$.
\end{proof}

\begin{cor}\label{gr4}
Compact Hausdorff groups are in $\mathcal C$.
\end{cor}
\begin{proof} Let $G$ be a compact Hausdorff group. By the Peter-Weyl theorem, $G$ is an inverse limit of compact Lie groups. Thus, in view of Proposition \ref{gr2}, we may assume that $G$ is a compact Lie group.  By Tannaka's theorem (see, for example, \cite{Canotes}) $G$ is isomorphic to $\R$-points of an algebraic groups. Therefore, by Proposition \ref{gr3}, $G\in \mathcal C$.
\end{proof}
\begin{pro}\label{gr5}
 Let $\Lambda$ be a  finitely generated residually finite group. If $\Lambda \in \mathcal C$, then $\Lambda$ is {\textbf L}GR.
\end{pro}
\begin{proof} Let  $\phi:\Lambda\hookrightarrow \Gamma$ be a Grothendieck pair.
Since $\Lambda \in \mathcal C$, the induced map 
$$\phi^{\#}_{\Lambda}: \Hom (\Gamma, \Lambda)\to \Hom (\Lambda, \Lambda), \ \tau\mapsto \tau\circ \phi,$$  is a bijection. Therefore, we obtain that $\Lambda$ is a retract of $\Gamma$ and since both groups are residually finite and $\widehat \phi$ is an isomorphism, $\phi$ is an isomorphism.
\end{proof}
 \cref{gr1}, \cref{gr3} and \cref{gr5} prove    \cref{leftGR} in the introduction since the free groups and the surface groups are isomorphic to some arithmetic groups.
\begin{proof}[Proof of Corollary \ref{bohr}]
The two  statements are proved similarly. Let us prove the second one. 

By Proposition \ref{gr3}  the affine complex algebraic  groups are in $\mathcal C$. Since $A(\Lambda)$ is proaffine complex algebraic  group, Proposition \ref{gr2} implies that $A(\Lambda)\in \mathcal C$. Thus, there exists $\tau:\Gamma\to A(\Lambda)$, $\alpha=\tau\circ \phi$, where $\alpha:\Lambda \to A(\Lambda)$ is the canonical map. The uniqueness of proalgebraic completion $\Gamma\to A(\Gamma)$ implies that $A(\phi)$ is an isomorphism.
\end{proof}

 \section{  Grothendieck left/right  rigid groups}\label{LRGR}
 In this section we discuss Grothendieck left/right rigid groups. \cref{leftGR} gives us a number of {\textbf L}GR groups, some of them, for example, finitely generated free groups and surface groups, are also {\textbf R}GR. In fact all LERF groups  are {\textbf R}GR. Recall that a group is called {\bf locally extended residually finite} (LERF for short) if every finitely generated subgroup is   closed in the profinite topology.  
 \begin{pro}
 Let $\Gamma$ be a finitely generated residually finite LERF group. Then $\Gamma$ is {\textbf R}GR.
 \end{pro}
 \begin{proof}
 If  $\phi:\Lambda\hookrightarrow \Gamma$ is a proper embedding and $\Lambda$  is finitely generated, then there exists a finite quotient of $\Gamma$, where the image of $\Lambda$ is a proper subgroup. Hence $\widehat \phi$ is not onto.
 \end{proof}
 This applies   to many finitely generated self-similar branch groups \cite{GLN21},  all lattices in $\SL_2(\R)$ and  in $\SL_2(\CC)$ \cite{Wi04, AFW15}. Not all  fundamental groups of  compact 3-manifolds are LERF. 
However, in   \cite{Su23} Sun proved that they are {\textbf R}GR. Furthermore, all non-uniform lattices in $\SL_2(\CC)$ are virtually free-by-cyclic  \cite{Wi04, AFW15}. These groups  are also {\textbf R}GR. In fact in Section \ref{ascending} we prove that  all ascending $HNN$-extensions of finitely generated free groups are {\textbf R}GR. 
Notice that by a result of Borisov and Sapir \cite{BS05}, the ascending $HNN$-extensions of finitely generated free groups are residually finite.

So in summary, all lattices in $\SL_2(\R)\cong \SO(2,1)$ are both {\textbf L}GR and {\textbf R}GR. For $\SL_2(\CC)\cong \SO(3,1)$ all are {\textbf R}GR and the arithmetic ones are also  {\textbf L}GR.
 \begin{Questions}\
 \begin{enumerate}
 \item[(a)] Is a non-arithmetic lattice in $\SL_2(\CC)$  {\textbf L}GR?
 \item[(b)]
 Is   the fundamental group of a compact 3-manifold {\textbf L}GR?
 \item[(c)] Are lattices in $\SO(n,1)$ {\textbf L}GR or {\textbf R}GR when $n\ge 4$?

 \end{enumerate}
 \end{Questions}

In spite all of this, there are hyperbolic groups which are not {\textbf R}GR.

   \begin{pro} \label{hyperbolic}
  There exists a hyperbolic group which is not  {\textbf R}GR. \end{pro}

 The proof uses the following fundamental result which plays important role in the construction of Grothendieck pairs    \cite{PT86,BL00, BG04}.
  \begin{teo} \label{criterionGP}
 Let $G$ be a finitely generated residually finite group and $N$ a normal subgroup. Assume that $G/N$ has no non-trivial finite quotients and $H_2(G/N;\Z)=0$.
\begin{enumerate}
\item[(a)]  If $N$ is finitely generated, then $N\hookrightarrow G$ is a Grothendieck pair.
\item [(b)] Let $P=\{(g_1,g_2)\in G\times G\colon g_1N=g_2N\}$. If $G/N$ is finitely presented, then $P\hookrightarrow  G\times G$ is a Grothendieck pair.
\end{enumerate} 
 \end{teo}
 A standard example of a non-trivial finitely presented group $Q$ without finite quotients and having trivial $H_2(Q;\Z)$ is the Higman group 
 \begin{equation} \label{Higman}
 \langle a,b,c,d:a^{-1}ba=b^2, b^{-1}cb=c^2,c^{-1}dc=d^2,d^{-1}ad=a^2\rangle.\end{equation}
 
 \begin{proof}[Proof of Proposition \ref{hyperbolic}]
 Using the construction of Rips \cite{Ri82}, we obtain that there are a finitely presented group $G$ and a normal subgroup $N$ of $G$ such that
 \begin{enumerate}
 \item[(i)] $G/N$ is the Higman group (\ref{Higman}).
 \item [(ii)] $G$ has a presentation satisfying the small cancellation condition $C^\prime(1/6)$.
 \item [(iii)] $N$ is finitely generated.
 \end{enumerate}
 The second condition implies that $G$ is of cohomological dimension 2, hyperbolic and can be cubulated (\cite{Wi04}). Thus  by \cite{Ag13},  $G$ is   virtually compact special, and so, residually finite. By  Theorem  \ref{criterionGP}, $ N\hookrightarrow G$ is a Grothendieck pair.
 \end{proof}

 Theorem \ref{criterionGP} also enables us to see that a {\textbf L}GR group is not always {\textbf R}GR.
 \begin{pro}\label{f2f2}
 There exist a {\textbf L}GR group which is not {\textbf R}GR.
 \end{pro}
 \begin{proof}
 Lett $F_4$ be  the free group of rank 4 and $\Gamma=F_4\times F_4$.  $\Gamma$ is  a finite index  subgroup of $\SL_2(\Z)\times \SL_2(\Z)$ and hence {\textbf L}GR by Corollary \ref{leftGR}.
 
Let $N$ be a normal subgroup of $F_4$ so that $F_4/N$ is the Higman group. Then, by Theorem \ref{criterionGP},   $\Gamma$ is not {\textbf R}GR   
 \end{proof}
 This is the example produced in \cite{PT86} as a first counter example to the Grothendieck question. More examples which are {\textbf L}GR and not {\textbf R}GR are  given in \cite{BL00}.   However we do not know whether {\textbf R}GR groups are always {\textbf L}GR.   \begin{Question}
  Is there an {\textbf R}GR group, which is not {\textbf L}GR?\end{Question}

 Observe that the subgroup $N$ in the proof of Proposition \ref{hyperbolic} is not finitely presented. Recall that a group is called {\bf coherent} if all its finitely generated subgroups are finitely presented. For example, the ascending $HNN$-extensions of finitely generated free groups  are  coherent  by  a result of Feighn and Handel \cite{FH99}.  For those we will show in Section \ref{HNN} that they are {\textbf R}GR .
 \begin{Question}
 Is a coherent residually finite  (hyperbolic) group {\textbf R}GR?
 \end{Question}

 Clearly if $G$ is a {\textbf L}GR or  {\textbf R}GR group, then so is a group that contains $G$ as a subgroup of finite index. But, we do not know whether the same conclusion holds for subgroups of finite index.
 \begin{Question}
 Is the property to be {\textbf L}GR or  {\textbf R}GR a commensurability invariant?
 \end{Question}
  
  \section{Ascending $HNN$-extensions of finitely generated free groups}\label{ascending}
  \label{HNN}
  Let $F$ be a finitely generated free group and $\alpha:F\to F$ an injective endomorphism of $F$. Then the group
$$M_\alpha=  \langle F,t |  t^{-1}ft=\alpha (f), f\in F\rangle $$
is called {\bf the mapping torus of $\alpha$} or {\bf ascending $HNN$ extension of $F$ corresponding
to $\alpha$}. We denote by $\kappa:F\to M_\alpha$ the canonical embedding of $F$ into $M_\alpha$ and by $j:M_\alpha\to\widehat{M_\alpha} $ the canonical homomorphism of $M_\alpha$ to its profinite completion.   If $H$ is a subgroup of $F$,  denote by  $\overline H$ be the closure of $j\circ \kappa(H)$ in $\widehat{M_\alpha}$.  We will also denote by $\widetilde H$ the closure of $H$ is $\widehat F$.

  The main result of this section is the following theorem.
 \begin{teo}\label{separation}
  Let $F$ be a finitely generated free group and $\alpha:F\to F$ an injective endomorphism of $F$. Let $H$ be a finitely generated  $\alpha$-invariant subgroup of $F$ and $w\in F$. Then $j\circ \kappa(w)\in \overline H$ if and only if there exists $n\in \N$ such that $\alpha^n(w)\in H$.
 \end{teo}
 The case $H=1$ of the theorem implies that $M_\alpha$ is residually finite (see \cite[Lemma 2.1]{BS05}). Our motivation is that this theorem implies the following result.
 \begin{teo}\label{GRHNN} An ascending HNN extension of a finitely generated free group  is {\textbf R}GR.
 \end{teo}
 
 Before proving Theorem  \ref{separation} and Theorem \ref{GRHNN} ,
 let us first describe the structure of $\widehat M_\alpha$. It was  previously  studied in   \cite[Theorem 5.8]{GJPZ14}.
\begin{pro}\label{profiniteHNN}
 The group  $\widehat M_\alpha$  is isomorphic to   $\widehat {\langle t\rangle}\ltimes P$, where  
\begin{enumerate}
\item  $P=\displaystyle \bigcap_{n\in \N} \widehat \alpha^n(\widehat F)$ and  
\item the restriction of $\widehat \alpha$ on $P$ is the automorphism that defines the semidirect product  $\widehat {\langle t\rangle}\ltimes P$.
\end{enumerate}
\end{pro}
The following example illustrates the meaning of the proposition in a straightforward case.
\begin{exa}
Consider  $\alpha :\Z\to \Z$, $1\mapsto p$. Then $M_\alpha\cong\Z\ltimes  \Z[\frac 1p]$ and $P\cong \prod_{q\ne p}\Z_q$.
\end{exa}
\begin{proof} By a theorem of Nikolov and Segal \cite{NS07}, the profinite completion of a finitely generated profinite group is isomorphic to the group itself. In this proof we will not distinguish between them.

Let $V$ be a verbal open subgroup of $\widehat F$. Since $V$ is verbal $\widehat \alpha(V)\le V$. Thus,  $\widehat \alpha$ induces an endomorphism $$\overline \alpha: \widehat F/V\to \widehat F/V.$$
Since $\widehat F/V$ is finite, for any $f\in \widehat F$ we can define $\lim_{n\to \infty} \overline{\alpha}^{n!}(fV)$. Observe that verbal open subgroups form a base of neighborhoods of the identity of $\widehat F$. Thus, for any $f\in   F$ we can define
$$  i:  F\to P,\  f\mapsto \lim_{n\to \infty} {\alpha}^{n!}(f)\in P.$$  It is clear that $  \alpha\circ  i= i \circ  \alpha$. Observe that $i$ induces the homomorphism 
$$ \widehat  i:   \widehat F\to P,\  g\mapsto \lim_{n\to \infty} \widehat{\alpha}^{n!}(g)\in P.$$
Let  $g\in P$ and let $V$ be again a   verbal open subgroup of $\widehat F$. Then there exists $h\in \widehat F$ such that $g=\widehat \alpha^{|F:V|!}(h)$. Therefore, $\widehat i(gV)=gV$. Thus, $\widehat i$ fixes the elements of $P$, and so $\widehat i$ is a retract. In particular, $i(F)$ is dense in $P$.

Since, $  \alpha\circ  i= i \circ  \alpha$ there exists the map
$$\epsilon :M_\alpha\to \widehat {\langle t\rangle}\ltimes_{\widehat \alpha} P, \ t\mapsto t, \ \kappa(f)\mapsto i(f). $$
Observe that the diagram
$$\begin{array}{ccc}
M_\alpha & \xrightarrow {\epsilon} &\widehat {\langle t\rangle}\ltimes_{\widehat \alpha} P\\
 \uparrow_{\kappa}&&||\\
F&\xrightarrow{i}&\widehat {\langle t\rangle}\ltimes_{\widehat \alpha} P \end{array}$$ is commutative. Hence  the diagram
$$\begin{array}{ccc}
\widehat{M_\alpha} & \xrightarrow {\widehat{\epsilon}} &\widehat {\langle t\rangle}\ltimes_{\widehat \alpha} P\\
 \uparrow_{\widehat{\kappa}}&&||\\
\widehat F&\xrightarrow{\widehat i}&\widehat {\langle t\rangle}\ltimes_{\widehat \alpha} P
\end{array}$$ is also commuttaive. In particular, $\ker \widehat \kappa \le \ker \widehat i$.

We want to show that $\widehat \epsilon$ is an isomorphism.
It is onto since $i(F)$ is dense in $P$.  Let $K$ be an open  normal subgroup of   $\widehat{M_\alpha}$. Put $K(F)=F\cap ( j\circ \kappa)^{-1}( K )$. 
Then $\alpha$ induces an automorphism on $F/ K(F)$. Therefore, $\ker \widehat  i\le \widetilde {K(F)}$. 
Since $$\ker \widehat \kappa=\bigcap_{K\unlhd_o \widehat{M_\alpha}} \widetilde{K(F)},$$  $\ker \widehat   i\le\ker \widehat  \kappa$. Thus, $\ker \widehat   i=\ker \widehat  \kappa$ and so $\ker \widehat \epsilon\cap \im \widehat \kappa=\{1\}$. On the other hand, $\widehat{M_\alpha}\cong \widehat {\langle t\rangle}\ltimes \im \widehat \kappa$, and so $\ker \widehat \epsilon\le \im \widehat \kappa$. Hence  $\widehat  \epsilon$ is injective.
\end{proof}
We  denote the map $\epsilon$ from the previous proof by $\epsilon_\alpha$. Since we have the following commutative diagram
$$\begin{array}{ccc}
M_\alpha & \xrightarrow {j} &\widehat{M_\alpha}\\
&\searrow_{\epsilon_\alpha}&\downarrow_{\widehat{\epsilon_\alpha}}\\
&& \widehat {\langle t\rangle}\ltimes_{\widehat \alpha} P\end{array}$$
and $\widehat{\epsilon_\alpha}$ is an isomorphism, the map $\epsilon_\alpha:M_\alpha\to  \widehat {\langle t\rangle}\ltimes_{\widehat \alpha} P$ provides all information about the embedding of $M_\alpha$ into its profinite completion.

Let $H$ be a finitely generated $\alpha$-invariant subgroup of $F$. We denote by $\beta:H\to H$ the restriction of $\alpha$ on $H$. Then the subgroup of $M_\alpha$ generated by $t$ and $H$ is isomorphic to $M_\beta$. Moreover, the inclusion map $\phi:M_\beta \hookrightarrow M_\alpha$  is onto if and only if there exists $n\in\N$ such that $\alpha^n(F)\le H$.

 Proposition \ref{profiniteHNN} gives us a description of $\widehat{M_\beta}$: there exists a homomorphism  $\epsilon_\beta:\widehat M_\beta\to \widehat {\langle t\rangle}\ltimes Q$, such that  
\begin{enumerate}
\item  $Q=\displaystyle \bigcap_{n\in \N} \widehat \beta^n(\widehat H)$,
\item the restriction of $\widehat \beta$ on $Q$ is the automorphism that defines the semidirect product  $\widehat {\langle t\rangle}\ltimes Q$ and
\item $\widehat{\epsilon_\beta}$ is an isomorphism.
\end{enumerate}
Observe that since $F$ is LERF, $\widetilde H\cong \widehat H$ and so $Q$   can be seen as a subgroup of $P$. By our definition of $\beta$, the restriction of $\widehat \alpha$ on $Q$ is $\widehat \beta$. Hence we have a canonical embedding
$\widehat {\langle t\rangle}\ltimes_{\widehat \beta} Q  \hookrightarrow  \widehat {\langle t\rangle}\ltimes_{\widehat \alpha} P$.
This description allows us to understand also the map $\widehat \phi$.
\begin{pro}\label{subgroupHNN}
The following diagram is commutative.
$$\begin{array}{ccc}
\widehat{M_\beta} & \xrightarrow{\widehat \phi}& \widehat{M_\alpha}\\
\downarrow^{\epsilon_\alpha}&&\downarrow^{\epsilon_\alpha}\\
\widehat {\langle t\rangle}\ltimes_{\widehat \beta} Q& \hookrightarrow & \widehat {\langle t\rangle}\ltimes_{\widehat \alpha} P.
\end{array}$$
In particular,
\begin{enumerate}
\item The map $\widehat \phi$ is injective.
\item  The map $\widehat \phi$ is surjective if and only if $j\circ \kappa(F)\le \overline H$.
\end{enumerate}
 \end{pro}

\begin{proof}
The commutativity of the diagram follows directly from the definitions of $\epsilon_\alpha$ and $\epsilon_\beta$ and the two conclusions from the commutativity of the diagram.
\end{proof}
Now we are ready to prove Theorems \ref{separation} and \ref{GRHNN}.
\begin{proof}[Proof of Theorem \ref{separation}]
The ``if" direction is clear. Let us show the other direction.

 Using the isomorphism $\widehat{\epsilon_{\alpha}}$ we will identify the elements of $\widehat{M_\alpha}$ and $ \widehat {\langle t\rangle}\ltimes_{\widehat \alpha} P$. With this identification the elements of $\overline H$ correspond to the elements of $Q=\displaystyle \bigcap_{n\in N} \widehat \alpha^n(\widetilde H)$, and so, by the hypothesis of the theorem,  $\epsilon_\alpha (w)=i(w)\in Q$,  where $i$ is the map constructed in the proof of Proposition \ref{profiniteHNN} (here we see $F$ as a subgroup of $M_\alpha$ and forget about $\kappa$).

Let us first assume that $H$ is of finite index in $F$. Then,   $i(w)\in   \widetilde H$. Hence there exists $n$ such that $\alpha^{n!}(w)\in \widetilde H$,
 and so, since $\widehat H\cap F=H$,  $\alpha^{n!}(w)\in  H$.

Now assume that $H$ is arbitrary. Let $V$ be a verbal subgroup of $F$ of finite index. Then $HV$ is a $\alpha$-invariant subgroup of $F$ of finite index and $j\circ \kappa(w)\in \overline {VH}$. Thus, by above, there exists $n\in \N$ such that $\alpha^n(w)\in HV$. 

There exists a subgroup $U$ in $F$ of finite index such that $H$ is a free factor of $U$. We can find a verbal subgroup $V$ of $F$ of finite index such that $V\le U$. Then, by Kurosh's subgroup  theorem, $H$ is a free factor of $VH$. Since $HV$ is $\alpha$-invariant, one can replace $F$ by $HV$, and so  without loss of generality we can assume that $H$ is a free factor of $F$.

The  following argument   uses  ideas of the proof of  \cite[Theorem 1.2]{BS05}. The situation considered in  \cite[Theorem 1.2]{BS05} corresponds to the case where $H=\{1\}$.

Let  $\F$ be  the algebraic closure of a finite field. Let  $M_4=\Mat_{4\times 4}(\F)$ be the affine algebraic $\F$-variety corresponding to $4$ by $4$ matrices with  the ring of regular functions $\F[a_{i,j}|1\le i,j\le 4]$.

The group $H$  is a free factor of $F$. Let $x_1,\ldots, x_k$ be a set of free generators of $F$ such that the first $l$ elements $x_1,\ldots, x_l$ are free generators of $H$.  To  each point $p=(A_1,\ldots, A_k)\in M_4^k$ corresponding to $k$ invertible matrices over $\F$ we associate the representation $\tau_p:F\to \GL_4(\F)$ that sends $x_i$ to $A_i$. These points form 
an open (in the Zariski topology) subset $W$ of $M_4^k$.

Consider the ring of regular functions on $M_4^k$
$$R=\F[a_{i,j}^m|\ 1\le i,j\le 4, \ 1\le m\le k]$$ and let $X_m=(a_{ij}^m)\in \Mat_{4\times 4}(R)$.
Consider a  word $v\in F$ in a reduced form as a word in $\{x_i^{\pm 1}\}$. 
Let $X_v$ be the matrix over $R$ obtained from $v$ by substituting $x_i$ by $X_i$ and $x_i^{-1}$ by the adjoint matrix $\adj (X_i)$.  Observe that if $p\in W$ then $$X_v(p)=c_{p,v}\cdot \tau_p(v), \textrm{\ where \ } 0\ne c_{p,v}\in \F.$$

We put
$$ \Phi:M_4^k\to M_4^k, \ \left (X_1,\ldots X_k \right )\to \left  (X_{\alpha(x_1)},\ldots, X_{\alpha(x_k)}\right ).$$
Let $N_4\le M_4$ be the closed subvariety corresponding  to the matrices of type 
$$N_4=\left \{\left ( \begin{array}{cccc}
a_{1,1} & 0 & 0 & 0\\
0 &a_{1,1} & 0 & 0 \\
0 & 0 & a_{3,3} & a_{3,4}\\
0 & 0 &a_{4,3} & a_{4,4} \end{array}
\right ): a_{ij}\in \F\right\}.$$
  \begin{claim}
We have that $\Phi(N_4^l\times M_4^{k-l})\subseteq N_4^l\times M_4^{k-l}$.
\end{claim}
\begin{proof}
This follows from the fact that $H$ is $\alpha$-invariant, and so $X_{\alpha(x_i)}$ are matrices over the ring $\F[a_{i,j}^m|\ 1\le i,j\le 4,\  1\le m\le l]$.
\end{proof}
The dimension of the variety $N_4^l\times M_4^{k-l}$ is equal to $n=5\cdot l+16\cdot (k-l)$.
We denote by $V$ the closure of $\Phi^n(N_4^l\times M_4^{k-l})$ in $N_4^l\times M_4^{k-l}$ with respect to the Zariski topology. 

Let $S$ be the ring of regular functions on $N_4^l\times M_4^{k-l}$, i.e. the quotient of the algebra $R$ by the ideal generated by 
$$\{a_{1,1}^m-a_{2,2}^m, a_{1,2}^m, a_{2,1}^m, a_{3,j}^m, a_{j,3}^m, a_{4,j}^m, a_{j,4}^m\colon 1\le m\le l; j=1,2\},$$ and let $\overline {X_i}$ be the image of $X_i$ in $\Mat_4(S)$. Let $Q$ be the ring of fractions of $S$. Consider the representation $\tau: F\to \GL_4(Q)$ that sends $x_i$ to $\overline{X_i}$. Observe that  if  for any ring $R$ we put 
$$P(R)= \left \{\left ( \begin{array}{cccc}
a_{1,1} & 0 & 0 & 0\\
0 &a_{1,1} & 0 & 0 \\
0 & 0 & a_{3,3} & a_{3,4}\\
0 & 0 &a_{4,3} & a_{4,4} \end{array}
\right ): a_{i,j}\in R\right \},$$
then $\tau(H)\le P(Q)$. Note also that for every $p\in (N_r^l\times M_r^{k-l})\cap W$ and  $v\in F$
$ \tau(v)(p)=\tau_p(v).$

\begin{claim} \label{duda}
We have that  $$H= \{v\in F: \tau(v)\in   P(Q)\}.$$
\end{claim}

\begin{proof}
 We have to show that for every $v\not \in H$, there exists $p\in N_4^l\times M_4^{k-l}$ such that $\tau_p(v)\not \in P(\F) $. 
 
 Write $F=H*T$.
 and let $v=h_0t_1h_1\ldots h_{k-1}t_k h_k$, where $k>0$, with $h_i\in H$ for $i=0,\ldots, k$, $1\ne t_i\in T$ for $i=1,\ldots , k$ and $h_i\ne 1 $ for $i=1,\ldots , k-1$. We can also assume that $h_0=h_k=1$. Therefore
 $v= t_1h_1\ldots h_{k-1}t_k$. 
 
 We can find a homomorphism $f:H\to P(\F)$ such that $f(h_i)\not  \le Z(\GL_4(\F))$ for $i=1,\ldots, k-1$. Put $a_i=f(h_i)$. Let
 $$b=\left ( \begin{array}{cccc}
b_{1,1} & 0 & 0 & 0\\
0 &1 & 0 & 0 \\
0 & 0 & 1 & 0\\
0 & 0 &0 & 1\end{array}
\right ),$$
with $b_{1,1}\in \F\setminus \{0,1\}$. We consider the generalized word

$$u(t)=   bt_1a_1t_2 \ldots a_{k-1} t_kb^{-1}t_k^{-1}a_{k-1}^{-1}\ldots t_2^{-1}a_1^{-1}t^{-1}_1.$$
 By \cite[Theorem 5]{To84}, there exists a homomorphism $g:T\to   \GL_4(\F)$ such that 
 $$ bg(t_1)a_1g(t_2) \ldots a_{k-1} g(t_k)b^{-1}g(t_k)^{-1}a_{k-1}^{-1}\ldots g(t_2)^{-1}a_1^{-1}g(t_1)^{-1} \ne 1.$$ In particular, $g(t_1)a_1g(t_2) \ldots a_{k-1}g(t_k)$ does not commute with $b$, and hence,
 $$g(t_1)a_1g(t_2) \ldots a_{k-1}g(t_k) \not \in P(\F).$$
 Thus, the representation, that is equal to $f$ when restricted to $H$ and to $g$ when restricted to $T$,   does not send $v$ to $P(\F)$.
\end{proof}
We put $$Z=\{p\in V:   X_w(p)\in P(\F)  \}.$$
If $Z\ne V$, then by  \cite[Theorem 1.4]{BS05}, there exists a quasi-fixed point $p\in (V\setminus Z)\cap W$ for $\Phi$. Thus a  power (say $q$) of $\Phi$ fixes $p$ and,    as in \cite[Lemma 2.2]{BS05}, we can construct a map $\theta: M_\alpha \to \PGL_4(\F)\wr C_q$ with finite image.  Observe that this map separates $w$ and $H$, i.e. $\theta(w)\not \in \theta(H)$, because the image of $\tau_p(H)\le P(\F)$  and $\tau_p(w)=c_{p,w}^{-1}X_w(p)\not \in  P(\F)$. But, since $j\circ \kappa(w)\in \overline H$, this is impossible. Thus $Z=V$. 

Let $p\in (N_r^l\times M_r^{k-l})\cap W$. Then  we have
$$\tau(\alpha^n(w))(p)= \tau_p(\alpha^n(w))=c_{p,\alpha^n(w)} ^{-1}\cdot  X_{\alpha^n(w)}(p)= c_{p,\alpha^n(w)} ^{-1}\cdot   X_w(\Phi^n(p)).$$ Thus, 
for all $p\in  (N_r^l\times M_r^{k-l})\cap W$ we have that
$ \tau(\alpha^n(w))(p)\in P(\F)$, 
and so, since $(N_r^l\times M_r^{k-l})\cap W$ is dense in $N_r^l\times M_r^{k-l}$,
$  \tau(\alpha^n(w))\in P(Q)$.
Therefore, by Claim \ref{duda}, $ \alpha^n(w) \in  H$.  \end{proof}
  
  \begin{proof}[Proof of Theorem \ref{GRHNN}]  Let $F$ be a finitely generated free group and $\alpha:F\to F$ an injective endomorphism of $F$.  Let $\Gamma=M_\alpha$. Assume that $\phi:\Lambda \hookrightarrow \Gamma$ is a non-trivial Grothendieck pair.  
  Let $\pi:\Gamma\to \Z$ be the canonical projection on $\Z$. Then, since $\widehat \phi$  is onto,  $\pi\circ \phi(\Lambda)=\Z$. Thus, without loss of generality we can assume that $t\in \phi(\Lambda)$.
  It follows from \cite[Proposition 2.3]{FH99}, that there are finite subsets $A$ and  $B$   of $F$ such that  $\phi(\Lambda)=\langle t, A, B\rangle$ and the following holds \begin{enumerate}
  \item $\langle A,\alpha(A)\rangle =\langle A,B\rangle$;
\item with respect to the generators $t, A, B$,  $\phi(\Lambda)$ has a presentation of   the following form $\phi(\Lambda)=\langle t, A, B| C\rangle$, where 
  $C=\{tat^{-1}\alpha(a^{-1}):a\in A\}$.
\end{enumerate}
Since the profinite completions of $\Lambda$ and $\Gamma$ are isomorphic, the L\"uck approximation \cite{Lu94} implies (see, for example, \cite[Corollary 6.4]{ReAndrews}) that  their first  $L^2$-Betti numbers coincide. The first $L^2$-Betti number of $\Gamma$ is zero. 

Since the deficiency of an infinite finitely presented groups minus 1 is at least the first $L^2$-Betti number (see \cite[theorem 3.21]{Ka19}),  $B$ is empty. Thus, $\Lambda$ is isomorphic to the ascending $HNN$-extension of the free group $H=\langle A\rangle$ corresponding to the restriction of $\alpha$ on $H$.  

By Proposition \ref{subgroupHNN},  $j\circ \kappa(F)\le \overline H$. By Theorem \ref{separation} there exists $n$ such that $\alpha^n(X)\subset H$  where $X$ is a finite generating set of $F$. Hence $\alpha^n(F)\le H$, i.e $t^{n}Ht^{-n}$ contains $F$, and so $F\le \phi(\Lambda)$.  Hence $\phi$ is an isomorphism.
 \end{proof}
 \section{Grothendieck properties versus profinite properties} 
 \label{profinite}

In \cite{Ta87} Tavgen\textquotesingle\  constructed a non-trivial  Grothendieck pair $\phi:\Lambda \hookrightarrow \Gamma$ of soluble groups such that $\Lambda\cong \Gamma$. For the convenience of the reader we recall   this ingenous construction leaving the reader to complete the details or consult \cite{Ta87}.

Let $p$ be a prime. For any $m\in \Z$, write $m=r+k(p-1)$, where $0<r\le p-1$ and put $q_m=r+kp$. 
Observe that $q_m$ is coprime with $p$. Recall that $\Z_{(p)}$  denote the localization of $\Z$  at the ideal $(p)$. 
Then the group $\Gamma$ is isomorphic to a semidirect product $$\Gamma=\left (\bigoplus_{l,m\in \Z} \Z_{(p)} \right )\rtimes \left  (\Z\oplus (\Z\wr \Z) \right ),$$
 where the action of the right subgroup on the left subgroup is described in the following way:

\begin{enumerate}
\item $ \Z\oplus (\Z\wr \Z)=\langle b,c_n, d| [b,c_n]=1, [b,d]=1, dc_nd^{-1}=c_{n+1}, n\in \Z\rangle$.
\item the element $a\in\Z_{(p)}$ of the $(l,m)$ summand of $\displaystyle \bigoplus_{l,m\in \Z} \Z_{(p)}$ is denoted by $a_{l,m}$.
\item $ba_{m,l}b^{-1}=a_{m,l-1}$, $c_na_{m,l}c_n^{-1}=q_{m+n}a_{m,l}$ and $da_{m,l}d^{-1}=a_{m-1,l}$.
\end{enumerate}

In order to construct $\Lambda$ consider the map $$f:\displaystyle \bigoplus_{l,m\in \Z} \Z_{(p)}\to \displaystyle \bigoplus_{m\in \Z} \Q,\ a_{m,l}\mapsto (p^la)_m,$$ where  the element $a\in\Q$ of the $m$th summand of $\displaystyle \bigoplus_{m\in \Z} \Q$ is denoted by $a_m$. Then $\ker f$ is a normal subgroup of $\Gamma$ and we define
 $$\Lambda=\ker f \rtimes \left  (\Z\oplus (\Z\wr \Z) \right ).$$
It seems plausible that $\Gamma$ in the example of Tavgen\textquotesingle\ is profinitely rigid.  
\begin{Question}
Is there a finitely generated residually finite  profinitely rigid group which is not Grothendieck left/right rigid?
\end{Question}

 An example in reverse is easy to construct. Let 
$C_{11}=\langle a\rangle $ be a cyclic group  of order 11,  $\alpha \in \Aut(C_{11}) $ a generator of $\Aut(C_{11})$ and $\phi_k:\Z\to \Aut(C_{11})$ that sends $1$ to $\alpha^k$. Then the group $\Z\ltimes _{\phi_2} C_{11}$ is clearly   Grothendieck left/right rigid. However, it is not profinitely rigid since its profinite completion is isomorphic to the one of   $\Z\ltimes _{\phi_6} C_{11}$ and the groups  $\Z\ltimes _{\phi_2} C_{11}$ and $\Z\ltimes _{\phi_6} C_{11}$ are not isomorphic.

Clearly for $\Lambda \hookrightarrow \Gamma$ to be a Grothendieck pair is stronger than just $\widehat \Lambda \cong \widehat \Gamma$. 
The following result shows that, indeed, some properties of Grothendieck pairs do not hold for groups having isomorphic profinite completions.
\begin{pro}\label{orth}
Consider the following two quadratic forms
\begin{multline*}
q_1= x_1^2+x_2^2+x_3^2+x_4^2+x_5^2+x_6^2-\sqrt 2 x_7^2-\sqrt 2 x_8^2 \textrm{\ and\ }\\ q_2=x_1^2+x_2^2-x_3^2-x_4^2-x_5^2- x_6^2-\sqrt 2 x_7^2-\sqrt 2 x_8^2.\end{multline*}
  Then there exist finite-index subgroups $\Gamma \le \Spin(q_1)(\Z[\sqrt 2])$ and $\Lambda  \le \Spin(q_2)(\Z[\sqrt 2]) $, with  $ \widehat \Gamma  \cong \widehat \Lambda$,  while their Bohr compactifications are not isomorphic.
\end{pro}

\begin{proof}Let us also define the quadratic forms
\begin{multline*}
\overline{q_1}= x_1^2+x_2^2+x_3^2+x_4^2+x_5^2+x_6^2+\sqrt 2 x_7^2+\sqrt 2 x_8^2 \textrm{\ and\ }\\ \overline{q_2}=x_1^2+x_2^2-x_3^2-x_4^2-x_5^2- x_6^2+\sqrt 2 x_7^2+\sqrt 2 x_8^2.\end{multline*}
For $i=1,2$, let  $G_i= \Spin(q_i)(\Z[\sqrt 2])$.
Let $\sigma_j:\Z[\sqrt 2]\to \R$ ($j=1,2$) two distinct
embeddings of   $\Z[\sqrt 2]$ into $\R$ such that $\sigma_1(\sqrt 2)=\sqrt 2$ and $\sigma_2(\sqrt 2)=-\sqrt 2$. They induce natural embeddings $$\sigma_{i,1}: G_i \hookrightarrow  \Spin(q_i)(\R)\textrm{\ and\ } \sigma_{i,2}: \Spin(q_i)(\Z[\sqrt 2]) \hookrightarrow \Spin(\overline{q_i})(\R).$$
Then we have two embeddings
$$ G_i\hookrightarrow  \Spin(q_i)(\R)\times   \Spin(\overline{q_i})(\R), \ x\mapsto(\sigma_{i,1}(x),\sigma_{i,2}(x)).$$
It is well known that these embeddings realize $\Spin(q_i)(\Z[\sqrt 2])$ as irreducible lattices.
By \cite[11.1]{Kn79} the congruence kernels of $G_i$ are finite and by \cite[Corollary 4]{Ak12}, the congruence completions of $G_i$ are isomorphic. Thus $\widehat G_1$ and $\widehat G_2$ are commensurable. Therefore,  there exist finite-index subgroups $\Gamma \le G_1$ and $\Lambda  \le G_2$
such that the profinite completion of $\Gamma$ is isomorphic to the profinite completion of $\Lambda$.

By \cite[Theorem 3]{Be23}, the connected component of the Bohr compactification of $\Gamma$ is $ \Spin(\overline{q_1})(\R)$ and  the connected component of the Bohr compactification of $\Lambda$ is trivial. \end{proof}

 Let $\mathcal P$ be a property of groups. We say that $\mathcal P$ is {\bf profinite} if for two finitely generated residually finite groups $\Lambda$ and $\Gamma$, 
 having the same profinite completion  the following holds: if $\Lambda$  satisfies $\mathcal P$, then $\Gamma$ also does.
In the last two decades, a lot of efforts were devoted to understand which properties of groups are profinite. Most of them lead to 
negative results  \cite{Ak12, Lu14, KKRS20,  KS23, Br23, EK23}. It is interesting to study which properties are shared by groups in Grothendieck pairs. We say that the property $\mathcal P$ is {\bf  up Grothendieck} if for every Grothendieck pair $\Lambda \hookrightarrow \Gamma$  if $\Lambda$  satisfies $\mathcal P$, then $\Gamma$ also does. The {\bf  down Grothendieck} property is defined symmetrically. In \cite{CB13} these two notions were called up/down weak profinite properties.

Clearly, every profinite property is both  up Grothendieck and  down Grothendieck. But the converse is not true. Corollary \ref{bohr} and Proposition \ref{orth} show that having a specific Bohr compactification is an   up/down  Grothendieck property but not a profinite one. Now, what is the relation between  down and up Grothendieck   properties?  Clearly every property of groups which is inherited by subgroups is down  Grothendick, e.g.   being  residually-$p$,  being  linear or being amenable. By \cite[Theorem 10.2]{KS23}, amenability is not  up Grothendieck. We expect that there are many properties of this sort that are not  up Grothendieck.
\begin{Questions}\
\begin{enumerate}
\item[(a)]
Is the property to be residually-$p$   up Grothendieck? 
\item [(b)]
Is the property to be linear   up Grothendieck? 
\end{enumerate}
\end{Questions}
The property to be residually-$p$ is not a profinite property by \cite{Lu14}, but we do not know whether the property to be linear is profinite. 

Finally, we notice that to be {\bf L}GR is an   up Grothendieck property and as we have seen in Proposition \ref{f2f2}, it is  not   down Grothendieck. Also there are less obvious examples. The property $(\tau)$ is    up Grothendieck; however, by \cite[Theorem 10.2]{KS23} it is not   down Grothendieck.

\section{Additional questions}

By analogy with  profinite rigidity it is natural also to consider   Bohr rigidity and  proalgebraic rigity.  A finitely generated residually finite group $\Lambda$ is called {\bf Bohr  rigid} ({\bf proalgebraic rigid})  if for any finitely generated residually finite group $\Gamma$ whose Bohr compactification (proalgebraic completion) is isomorphic to the one of $\Lambda$, $\Gamma\cong \Lambda$.  In the case of Grothendieck pairs, the ambient groups have the isomorphic Bohr compactifications.

The profinite completion $\widehat \Gamma$ of $\Gamma$ is isomorphic to the $\B(\Gamma)$ modulo its connected componenct, so if $\B(\Gamma)\cong \B(\Lambda)$, then $\widehat \Gamma\cong \widehat \Lambda$. Hence   profinite rigidity implies   Bohr rigidity. In the same way profinite rigidity implies proalgebraic rigidity.
\begin{Questions}\
\begin{enumerate}
\item[(a)] Is a finitely generated residually finite  group which is    Bohr rigid also   profinite rigid?
\item[(b)] Is a finitely generated residually finite  group group which is proalgebraic rigid also profinite rigid?
\end{enumerate}
\end{Questions}

A long-standing  question of Remeslennikov asks whether a finitely generated free group is profinitely rigid. The solution to this problem is far from being achieved by the methods that we have. Perhaps these weaker questions may be easier to handle.\begin{Question}
Is a finitely generated free group     Bohr rigid? Is a finitely generated free group     proalgebraic rigid?
\end{Question}

 \end{document}